\pdfoutput=1
\RequirePackage{ifpdf}
\ifpdf % We are running pdfTeX in pdf mode
\documentclass[pdftex]{sigma}
\else
\documentclass{sigma}
\fi

\numberwithin{equation}{section}

\newtheorem{Theorem}{Theorem}[section]
\newtheorem{Corollary}[Theorem]{Corollary}
\newtheorem{Lemma}[Theorem]{Lemma}
{ \theoremstyle{definition}
\newtheorem{Definition}[Theorem]{Definition}

\newtheorem{Remark}[Theorem]{Remark} }

\newcommand{\ve}{\varepsilon}
\newcommand{\bb}{\mathbf{b}}
\newcommand{\R}{\mathbb{R}}
\newcommand{\cC}{\mathcal{C}}
\newcommand{\Sph}{\mathbb{S}}

\begin{document}

\allowdisplaybreaks

\newcommand{\arXivNumber}{1504.03921}

\renewcommand{\PaperNumber}{036}

\FirstPageHeading

\ShortArticleName{The Co-Points of Rays are Cut Points of Upper Level Sets for Busemann Functions}

\ArticleName{The Co-Points of Rays are Cut Points\\ of Upper Level Sets for Busemann Functions}

\Author{Sorin V.~SABAU}

\AuthorNameForHeading{S.V.~Sabau}

\Address{School of Science, Department of Mathematics, Tokai University, Sapporo 005--8600, Japan}
\Email{\href{mailto:sorin@tokai.ac.jp}{sorin@tokai.ac.jp}}

\ArticleDates{Received August 07, 2015, in f\/inal form April 06, 2016; Published online April 13, 2016}

\Abstract{We show that the co-rays to a ray in a complete non-compact Finsler manifold contain geodesic segments to upper level sets of Busemann functions. Moreover, we cha\-rac\-terise the co-point set to a ray as the cut locus of such level sets. The structure theorem of the co-point set on a surface, namely that is a local tree, and other properties follow immediately from the known results about the cut locus. We point out that some of our
f\/indings, in special the relation of co-point set to the upper lever sets, are new even for Riemannian manifolds.}

\Keywords{Finsler manifolds; ray; co-ray (asymptotic ray); cut locus; co-points; distance function; Busemann function}

\Classification{53C60; 53C22}

\section{Introduction}

Roughly speaking, a Busemann function is a function that measures the distance to a point at inf\/inity on a complete boundaryless non-compact Riemannian or Finsler manifold. Originally introduced by H.~Busemann for constructing a theory
of parallels for straight lines (see \cite{Bu1,In1,In2,Sh}), the function plays a fundamental role in the study of complete non-compact Riemannian or Finsler manifolds (see \cite{Oh,Sh,SST}, etc).

In the present paper, we study the dif\/ferentiability of the Busemann function in terms of co-rays and co-points to a ray in the general case of a forward complete non-compact Finsler manifold. We show that the notions of geodesic segments to a closed subset and the cut locus of such sets can be extremely useful in the study of co-rays and co-points to a ray, that is points where Busemann function is not dif\/ferentiable.

The originality of our research is two folded. Firstly, the detailed study of Busemann functions, co-rays and co-points on Finsler manifolds is new. Secondly, in the special case of Riemannian manifolds, our main Theorems~\ref{th1.4} and~\ref{th1.7}, f\/irst statement, are new and they lead to new elementary proofs of other results already known.

Let $(M,F)$ be a forward complete boundaryless Finsler manifold. A~unit speed globally minimising geodesic $\gamma\colon [0,\infty)\to
M$ is called a {\it $($forward$)$ ray}. A ray $\gamma$ is called
{\it maximal} if it is not a proper sub-ray of another ray, i.e.,
for any $\ve>0$ its extension to $[-\ve,\infty)$ is not a~ray
anymore. Moreover, let us assume that $(M,F)$ is bi-complete, i.e., forward and backward complete.
A Finslerian unit speed globally minimising geodesic $\gamma\colon \R\to
M$ is called a {\it straight line}. We point out that, even though for def\/ining rays and straight lines we not need any completeness hypothesis, without completeness, introducing rays and straight lines would be meaningless.

Let $\gamma\colon [0,\infty)\to M$ be a given forward ray and let
$x$ be a point on a non-compact forward complete
Finsler manifold $(M,F)$.
Then, a forward ray $\sigma\colon [0,\infty)\to M$
is called a
{\it forward co-ray $($or a forward asymptotic ray$)$ to $\gamma$} if
 there exists a sequence of minimal geodesics $\{\sigma_j\}$ from $q_j:=\sigma_j(0)$ to $\sigma_j(l_j):=\gamma(t_j)$, for some divergent sequence of numbers $\{t_j\}$, such that $\lim\limits_{j\to \infty}q_j=\sigma(0)$ and $\dot{\sigma}(0)=\lim\limits_{j\to \infty}\dot{\sigma}_j(0)$.

A co-ray to $\gamma$ is called {\it maximal} if
for any $\ve>0$ its extension to $[-\ve,\infty)$ is not an
co-ray to~$\gamma$ anymore.
The origin points of maximal co-rays of~$\gamma$ are
called the {\it co-points} to~$\gamma$ (a slightly stronger def\/inition can be found in~\cite{Oh}).

Similarly, one can def\/ine {\it asymptotic straight lines}. If $\gamma\colon \R\to M$ is a straight line in a bi-complete Finsler manifold, then the straight line $\sigma\colon \R\to M$ is an asymptotic straight line to $\gamma$ if $\sigma |_{[0,\infty)}$ is asymptotic ray to $\gamma |_{[0,\infty)}$, and $\bar{\sigma} |_{[0,\infty)}$ is asymptotic ray to $\bar{\gamma} |_{[0,\infty)}$ with respect to the {\it reverse Finsler metric} $\bar F(x,y):=F(x,-y)$, where $\bar{\sigma}(t):=\sigma(-t)$ and $\bar{\gamma}(t):=\gamma(-t)$, $t\in [0,\infty)$ are the reverse rays of $\sigma$ and $\gamma$, respectively (see \cite{Oh} for details). This def\/inition makes sense because if $\sigma$ is a geodesic for $F$, then the reverse curve $\bar\sigma(t)$ is geodesic for $\bar F$.

If $\gamma$ is a forward ray in the forward complete boundaryless non-compact Finsler manifold $(M,F)$, then the {\it Busemann function} is def\/ined by
$x\in M\mapsto \bb_\gamma(x):=\lim\limits_{t\to\infty}\{t-d(x,\gamma(t))\}$, where $d$ is the Finsler distance function (see Section~\ref{sec: Busemann functions} for details).

Let us recall from~\cite{TS} some notions that will be useful later.

Let $N\subset M$ be a closed subset of $M$. For a point $p\in M{\setminus} N$, a unit speed geodesic segment $\alpha\colon [0,a]\to M$ from $p=\alpha(0)$ is called a {\it forward $N$-segment}
if $d(\alpha(t),N)=a-t$ holds on $[0,a]$, where $d(x,N):=\inf\{d(x,q)\colon q\in M\}$ for any $x\in M$. The existence of $N$-segments from any $p\in M{\setminus} N$ follows from the fact that $N$ is closed and the forward completeness hypothesis. If a unit speed (nonconstant) geodesic segment $\alpha\colon [0,a]\to M$ is maximal as an $N$-segment, then the point $p=\alpha(0)$ is called a {\it cut point} of $N$ along the $N$-segment $\alpha$, i.e.,
any geodesic extension $\tilde \alpha\colon [-\ve,a]\to M$, $\varepsilon>0$, $\tilde{\alpha}|_{[0,a]}={\alpha}|_{[0,a]}$ of $\alpha$ is not a forward $N$-segment anymore. The {\it cut locus} of $N$, denoted hereafter $\mathcal C_N$, is the set of all cut points of $N$ along all nonconstant $N$-segments. Observe that $\mathcal C_N\cap N=\varnothing$.
If a point $p\in M{\setminus} N$ admits two $N$-segments of equal length, then~$p$ is a cut point of~$N$. Therefore, any interior point of $N$-segment is not a cut point of~$N$.

We point out that in~\cite{TS}, for a closed subset $N$ of a backward complete Finsler manifold $(M,F)$,
a backward $N$-segment is def\/ined analogously. The notions of forward and backward $N$-segments to a closed subset $N$ are equivalent. Indeed,
if we consider the {\it reverse Finsler metric} $\widetilde F$ on the manifold $M$ given by $\widetilde F(x,y):=F(x,-y)$ for each
$(x,y)\in TM$, a backward $N$-segment on $(M,F)$ is a forward $N$-segment on $(M,\widetilde F)$.

Notice that, since we consider only boundaryless manifolds, any geodesic segment on a compact interval admits forward and backward local geodesic extensions even if the manifold $M$ is not forward nor backward complete. For more basics on Finsler manifolds see~\cite{BCS} or~\cite{S}.

Here are the main results of our paper.
\begin{Theorem}\label{th1.4}
Let $(M,F)$ be a forward complete boundaryless Finsler manifold and let $\alpha\colon [0,a]$ $\to M$ be a unit-speed geodesic.
The following three statements are equivalent.
\begin{enumerate}\itemsep=0pt
\item[$1.$]
$\alpha$ is a subarc of a co-ray to $\gamma$.
\item[$2.$]
$\alpha$ satisfies
\begin{gather}\label{eq1}
\bb_\gamma(\alpha(s))=s+\bb_\gamma(\alpha(0))
\end{gather}
 for all $s\in[0,a]$.
\item[$3.$]
$\alpha$ is a forward $N^b_\gamma$-segment, where $N^b_\gamma:=\bb_\gamma^{-1}[b,\infty)$ and $b=\bb_\gamma(\alpha(a))$.
\end{enumerate}
\end{Theorem}

From here the relation between co-points to a forward ray and the cut points of a level set of Busemann function naturally follows.

\begin{Theorem}\label{th1.7}
Let $(M,F)$ be a forward complete boundaryless Finsler manifold, and $\gamma$ a ray in~$M$.
\begin{enumerate}\itemsep=0pt
\item[$1.$]
For every $b\in \R$, the set of co-points of $\gamma$ in the sub-level $\bb_\gamma^{-1}(-\infty,b)$ is exactly the cut locus of the subset $N^b_\gamma$, i.e., $\mathcal C_\gamma\cap \bb_\gamma^{-1}(-\infty,b)={\mathcal C}_{N_\gamma ^b}$.
Moreover,
${\mathcal C}_{N_\gamma ^b}\subset {\mathcal C}_{N_\gamma ^c}$, for any $b<c$.

\item[$2.$] The Busemann function $\bb_\gamma$ is differentiable at a point $x$ of $M$ if and only if
$x$ admits a unique co-ray $\sigma$ to $\gamma$ emanating from $x=\sigma(0)$. In this case $\nabla\bb_\gamma(x)=\dot{\sigma}(0)$.
\end{enumerate}
\end{Theorem}

Loosely speaking, $\mathcal C_\gamma=\bigcup_{b}{\cal C}_{N_\gamma^b}$, where ${\mathcal C}_{N_\gamma ^b}$ denotes the cut locus of $N^b_\gamma$. Here ``loosely" means that it is possible that the local part of ${\cal C}_{N_\gamma^b}$ near a f\/ixed point $x$ keeps changing and never stabilises as $b$ goes to $+\infty$.

These two main theorems make possible to apply known results about the $N$-segments and cut points of a closed subset of $(M,F)$ to the study of co-rays and co-points, allowing the use of our previous results from \cite{TS}.

Seen in this light, the proof of the structure theorem for the co-point set on a Finsler surface, namely that is a local tree, becomes trivial.
It is also clear that the topology of
$(\mathcal C_\gamma,\delta)$, with the induced metric, coincides with the topology of the Finsler surface, as well as that $(\mathcal C_\gamma,\delta)$
is forward complete (see Theorem~\ref{th1.11}). Other results are also straightforward from~\cite{TS} (see Theorem~\ref{th1.12}).

Section~\ref{sec: diff of Bus} contains some consequences of the characterisation of the Busemann function's dif\/ferentiability given above. Here we study the conditions for the set $\bb^{-1}_\gamma(-\infty,c]$ to be compact (Theorem~\ref{th2.2}), and
for $\bb_\gamma$ to be an exhaustion (Corollary~\ref{exhaust}). If the co-point set $\mathcal C_\gamma$ contains an isolated point, then some important consequences are proved in Theorem~\ref{MTF2}.

\section{Busemann functions}\label{sec: Busemann functions}

Let $(M,F)$ be a forward complete boundaryless non-compact Finsler manifold (see \cite{BCS,S} for details on the completeness of Finsler manifolds).
In Riemannian geometry, the forward and backward completeness are equivalent, hence the words ``forward'' and ``backward'' are superf\/luous, but in Finsler geometry these are not equivalent anymore.

\begin{Definition}\label{def:Bus_fct}
If $\gamma\colon [0,\infty)\to M$ is a ray in a forward complete boundaryless non-compact Finsler manifold $(M,F)$, then the function
\begin{gather}\label{Bus_fnc_def}
\bb_\gamma\colon  \ M\to\R,\qquad
\bb_\gamma(x):=\lim_{t\to\infty}\{t-d(x,\gamma(t))\}
\end{gather}
is called {\it the Busemann function with respect to
$\gamma$}, where~$d$ is the Finsler distance function.
\end{Definition}

The Busemann function for Finsler manifolds
was introduced and partially studied by Eg\-lof\/f~\cite{Eg} and
more recently by~\cite{Oh}.

\begin{Remark}\quad
\begin{enumerate}\itemsep=0pt
\item The limit in~\eqref{Bus_fnc_def} always exists because the
function
$t\mapsto t-d(x,\gamma(t))$
is monotone nondecreasing and bounded above by $d(\gamma(0),x)$.

\item
Obviously $\bb_\gamma(\gamma(t))=t$, for all $t\geq 0$.
Moreover, if $\gamma_0$ is a sub-ray of the ray $\gamma$, then $\bb_{\gamma_0}(x)=\bb_\gamma(x)-t_0$ for any point $x\in M$, where $t_0\geq 0$ is the parameter value on $\gamma$ such that $\gamma_0(0)=\gamma(t_0)$.
\end{enumerate}
\end{Remark}

It follows that
a point $x$ of $M$ is an element of $\bb_\gamma^{-1}(a,\infty)$, for some real number $a$,
if and only if $t-d(x,\gamma(t))>a$ for some $t>0$, and
hence we get
\begin{Lemma}\label{lem1.1}
For each $a\in \R$,
$\bb^{-1}_\gamma(a,\infty)=\bigcup_{t>0}B_t^-(\gamma(t+a))$
holds, where $B_t^-(\gamma(t+a)):=\{x\in M \ |\ d(x,\gamma(t+a))<t \}$ denotes the backward open ball centred at $\gamma(t+a)$ of radius $t$.
In particular $\bb_\gamma^{-1}(a,\infty)$ is arcwise connected for each $a\geq 0$.
\end{Lemma}

The triangle inequality implies
\begin{Lemma}\label{lem: Lispchitzness}
 The function $\bb_\gamma$ is locally Lipschitz, i.e.,
\begin{gather*}
-d(x,y)\leq \bb_\gamma(x)-\bb_\gamma(y)\leq d(y,x)
\end{gather*}
for any two points $x,y\in M$.
\end{Lemma}

The dif\/ferentiability of Busemann function is fundamental for
the study of co-rays. Some results are already known (see for instance~\cite{In1}). Let us denote by $\nabla f(x)$ the Finslerian gradient of a smooth function $f\colon M\to \R$ (see~\cite{Oh} or~\cite[p.~41]{S}).

\begin{Theorem}[\cite{Oh}]\label{thm:Busemann functions properties by Ohta}
Let $\gamma$ be a forward ray in a non-compact forward complete
Finsler mani\-fold~$(M,F)$.
\begin{enumerate}\itemsep=0pt
\item[$1.$] For any $x\in M$, there exists at least one co-ray $\sigma$
of $\gamma$ such that $\sigma(0)=x$.
\item[$2.$] If the geodesic ray $\sigma$ is a co-ray to $\gamma$, then
$\bb_\gamma(\sigma(s))=s+\bb_\gamma(\sigma(0))$, $\forall\,
s\geq 0$.

\item[$3.$] If $\bb_\gamma$ is~differentiable at a point $x\in M$,
then $\sigma(s):=\exp_x(s\nabla \bb_\gamma(x))$ is the unique
co-ray to $\gamma$ emanating from $x$, where $\nabla \bb_\gamma(x)$ is the Finslerian gradient of
$\bb_\gamma$ at~$x$.
\end{enumerate}
\end{Theorem}

\begin{Remark}
The converse of statement~2 in Theorem~\ref{thm:Busemann functions properties by Ohta} is actually contained in our Theorem~\ref{th1.4}, 2~$\Rightarrow$~1.
\end{Remark}

For any closed subset $N$ of $M$, we have def\/ined $N$-segments in Introduction.
From now on, {\it any $N^b_\gamma$-segment will mean forward $N^b_\gamma$-segment}, where
$N^b_\gamma:=\bb^{-1}_\gamma [b,\infty)$.

\begin{proof}[Proof of Theorem~\ref{th1.4}]
$1\Rightarrow 2$.
Suppose that the property 1 holds.
Then, statement 2 follows immediately from Theorem~\ref{thm:Busemann functions properties by
Ohta}(2).

$2\Rightarrow 3$.
Choose any $s\in[0,a]$ and any $x\in \bb_\gamma^{-1}[b,\infty)$, where $b:=\bb_\gamma(\alpha(a))$. By def\/inition we have $\bb_\gamma(x)\geq b=\bb_\gamma(\alpha(a))$, and
from Lemma~\ref{lem: Lispchitzness}
it follows
\begin{gather}
\label{eq2}
\bb_\gamma(\alpha(a))-\bb_\gamma(\alpha(s))\leq \bb_\gamma(x)-\bb_\gamma(\alpha(s))\leq
d(\alpha(s),x).
\end{gather}

On the other hand, the relation~\eqref{eq1} implies
\begin{gather}\label{eq3}
d(\alpha(s),\alpha(a))\leq a-s=\bb_\gamma(\alpha(a))-\bb_\gamma(\alpha(s)).
\end{gather}

From relations \eqref{eq2} and \eqref{eq3}
it results
$d(\alpha(s),\alpha(a))=d(\alpha(s),N^b_\gamma)$
for any $s\in[0,a]$, and
since the point $x$ is arbitrarily chosen from $N^b_\gamma$
we obtain that $\alpha$ is an $N^b_\gamma$-segment.

$3\Rightarrow 1$.
Choose any suf\/f\/iciently small $\varepsilon>0$.
Let $\sigma_\varepsilon \colon [\varepsilon,\infty)\to M$ denote a co-ray to $\gamma$ emanating from $\alpha(\varepsilon)$,
$\alpha|_{(\varepsilon,\infty)}\neq\sigma_\varepsilon|_{(\varepsilon,\infty)}$.

Since $\sigma_\varepsilon$ satisf\/ies \eqref{eq1} for all $s\geq \varepsilon$,
$\sigma |_{[\varepsilon,a]} $ is also an $N^b_\gamma$-segment emanating from $\alpha(\varepsilon)$.
Thus, the two geodesic segments
$\alpha|_{[\varepsilon,a]}$ and $\sigma_\varepsilon|_{[\varepsilon,a]}$ must coincide,
since $\alpha(\varepsilon)$ is an interior point of $\alpha$ and interior points of $N$-segments have a unique $N$-segment.
Therefore, $\alpha$ is a subarc of the co-ray $\lim\limits_{\varepsilon\to 0}\sigma_\varepsilon$.
\end{proof}

By Theorem~\ref{th1.4} we get
\begin{Corollary}\label{cor1.5}
If a unit speed geodesic $\sigma\colon [0,a]\to M$
satisfies relation \eqref{eq1}, for all $s\in [0,a]$, then $\sigma$ is a co-ray to $\gamma$.
\end{Corollary}

\begin{Corollary}\label{cor1.6}
For each $a\in \R$ such that $\bb_\gamma^{-1}(a)\ne\varnothing$, we have
\begin{gather*}
d(x,N^a_\gamma)=a-\bb_\gamma(x),\qquad \forall\, x\in \bb_\gamma^{-1}(-\infty,a].
\end{gather*}
Hence, $\bb_\gamma$ is differentiable at a point~$x$ if and only if
for each real number $a>\bb_\gamma(x)$ the distance function $d(\cdot,N^a_\gamma)$ is differentiable at~$x$.
\end{Corollary}

\begin{proof}
Choose any $x\in \bb_\gamma^{-1}(-\infty,a]$, and
denote by $\sigma\colon [0,\infty)\to M$ a co-ray to $\gamma$ emanating from $x=\sigma(0)$.
Since $\sigma|_{[0,a-\bb_\gamma(x)]}$
is an $N^a_\gamma$-segment and noticing that $\sigma(a-\bb_\gamma(x))\in \bb_\gamma^{-1}(a)$,
 we obtain
$d(x,N^a_\gamma)=d(\sigma(0),\sigma(a-\bb_\gamma(x)))=a-\bb_\gamma(x)$.
\end{proof}

\begin{proof}[Proof of Theorem~\ref{th1.7}]
Let $x\in M$ be any co-point of $\gamma$. Choose $b>0$ such that $b>\bb_\gamma(x)$.

Then, from Theorem~\ref{th1.4} it follows that, for any co-ray $\sigma \colon [0,\infty)\to M$ of $\gamma$ from $x$,
we have
\begin{itemize}\itemsep=0pt
\item the relation $\bb_\gamma(\sigma(s))=s+\bb_\gamma(\sigma(0))=s+\bb_\gamma(x)$ holds good for any $s\geq 0$. Hence, for our chosen $b>0$, there always exists $a>0$ such that $b=\bb_\gamma(\sigma(a))=a+\bb_\gamma(x)$;
\item the geodesic segment $\sigma|_{[0,a]}$ is a maximal $N^b_\gamma$-segment.
\end{itemize}
 It results that $x\in \mathcal C_{N_\gamma^b}$.

Conversely, we choose any point $x\in \cup_{b>0}\mathcal{C}_{N_\gamma^b}$. It follows that $x$ is a cut point of ${N_\gamma^b}$, for some $b>0$.
 Let $\sigma\colon [0,a]\to M$ be an $N^b_\gamma$-segment from $x=\sigma(0)$, where $b=\bb_\gamma(\sigma(a))=a+\bb_\gamma(x)$. Theorem~\ref{th1.4} implies that there exists a maximal co-ray $\widetilde{\sigma}\colon [0,\infty)\to M$ of $\gamma$ such that $ \widetilde{\sigma}|_{[0,a]}={\sigma}|_{[0,a]}$. This means that $x=\widetilde{\sigma}(0)\in \mathcal{C}_\gamma$.

 We will prove now that ${\mathcal C}_{N_\gamma ^b}\subset {\mathcal C}_{N_\gamma ^c}$, for any $b<c$. Indeed, let us consider any point $x\in {\mathcal C}_{N_\gamma ^b}$, and let $\sigma\colon {[0,a]}\to M$ be an $N^b_\gamma$-segment emanating from $x=\sigma(0)$, i.e., $b=\bb_\gamma(\sigma(a))=a+\bb_\gamma(x)$.

 Notice that any short backward geodesic extension $\widetilde{\sigma}\colon {[-\varepsilon,a]}\to M$ of $\sigma$, where $\varepsilon>0$, cannot be an $N^b_\gamma$-segment, due to the fact that $x\in {\mathcal C}_{N_\gamma ^b}$.

 On the other hand, by Theorem~\ref{th1.4}, $\sigma|_{[0,a]}$ is a subarc of a maximal co-ray $\widetilde{\sigma}\colon {[0,\infty)}\to M$ of $\gamma$. Taking into account that $a+\bb_\gamma(x)=b<c$, that is, there exists $\widetilde{a}>0$ such that $a<\widetilde{a}:=c-\bb_\gamma(x)$, then again from Theorem~\ref{th1.4} it results that the subarc $\widetilde{\sigma}|_{[0,\widetilde{a}]}$ is a maximal $ {\mathcal C}_{N_\gamma ^c}$-segment, and hence $x\in {\mathcal C}_{N_\gamma ^c}$. In other words we have proved that ${\mathcal C}_{N_\gamma ^b}\subset {\mathcal C}_{N_\gamma ^c}$, for any $b<c$.

2. Follows easily from Theorem~A in~\cite{TS}, Corollary~\ref{cor1.6} and Theorem~\ref{th1.4}.
\end{proof}

\begin{Corollary} \label{cor1.9}
If $x\in M$ is an interior point of a co-ray $\sigma$ of
$\gamma$, then $\bb_\gamma$ is differentiable at $x$.
\end{Corollary}

\begin{proof}
Choose any point $\sigma(t_0)$, $t_0>0$.
By Theorem~\ref{th1.4},
the subray $\sigma|_{[t_0,\infty)}$ is a unique co-ray to $\gamma$ emanating from
$\sigma(t_0)$. Thus, Theorem~\ref{th1.7}, statement~2 shows that~$\bb_\gamma$ is dif\/ferentiable at~$\sigma(t_0)$.
\end{proof}

Let us denote by $\cC_\gamma$ the co-point set of the ray
$\gamma$, that is the origin points of the {\it maximal}
co-rays to $\gamma$.
\begin{Remark}
From the def\/inition of co-points it follows that if $p\in \mathcal C_\gamma$, then there exists a~maximal co-ray of $\gamma$ emanating from $p$. Equivalently, any co-ray emanating from $p\in \mathcal C_\gamma$ is maximal.
\end{Remark}

By Proposition 2.5 in \cite{TS} and our Theorem~\ref{th1.4} we obtain the following.

\begin{Corollary}\label{cor1.10}
Let $(M,F)$ be a forward complete boundaryless Finsler manifold, $\gamma$ a~forward ray in $M$ and $\cC_\gamma$ the co-point set of
$\gamma$.

Then, the subset
\begin{gather*}
\cC_\gamma^{(2)}:=\{p\in \cC_\gamma\colon
\textrm{there exist at least two maximal co-rays
from } p \textrm{ to } \gamma\}\subset \cC_\gamma
\end{gather*}
is dense in $\cC_\gamma$.
\end{Corollary}

\begin{Remark}
Let ${\mathcal{ND}(\bb_\gamma)}\subset M$ be the set of non-dif\/ferentiable points of the Busemann function $\bb_\gamma$. Then,
from Corollaries~\ref{cor1.9} and~\ref{cor1.10} it follows
$\cC_\gamma^{(2)}={\mathcal{ND}(\bb_\gamma)}\subset \mathcal C_\gamma\subset \overline{\mathcal{ND}(\bb_\gamma)}$.

In the special case when $\cC_\gamma$ is closed set in $M$, we have ${\mathcal{ND}(\bb_\gamma)}\subset \mathcal C_\gamma= \overline{\mathcal{ND}(\bb_\gamma)}$. This is not true in general (see Remark~\ref{rm: not closed}).
\end{Remark}

In the two dimensional case, the structure theorems of the cut locus from \cite{TS} can be easily extendend to the structure of $\mathcal C_\gamma$.
We recall that an injective continuous map from the open interval $(0,1)$ (or closed interval $[0,1]$) of $\mathbb R$ and from a circle $\Sph^1$ into
 $M$ is called a {\it Jordan arc} and a {\it Jordan curve}, respectively.

A topological space $T$ is called a {\it tree} if
 any two points in $T$ can be joined by a unique Jordan arc in $T$.
Likewise, a topological space $C$ is called a {\it local tree} if for every point
 $x\in C$ and for any neighborhood $U$ of $x$,
 there exists a neighborhood $V\subset U$ of $x$ such that $V$
 is a tree.

 A continuous curve $c \colon [a,b]\to M$
is called {\it rectifiable} if its length
\begin{gather*}%\label{d-length}
l(c):=\sup\left\{\sum_{i=1}^{k}\;d(c(t_{i-1}),c(t_i)) \, | \, a=:t_0<t_1<\cdots<t_{k-1}<t_k:=b\right\}.
\end{gather*}
is f\/inite.

By Theorem~\ref{th1.7} and Theorem~B in~\cite{TS} we obtain (compare with~\cite{L})
\begin{Theorem}\label{th1.11}
Let $\gamma$ be a ray in a forward complete boundaryless $2$-dimensional Finsler mani\-fold~$(M,F)$.
Then, the of co-point set $\mathcal C_\gamma$
 of $\gamma$ satisfies the following three properties.
\begin{enumerate}\itemsep=0pt
\item[$1.$]
The set $\mathcal C_\gamma$ is a local tree and any two co-points on the same connected component of $\mathcal C_\gamma$ can be joined by a~rectifiable curve in $\mathcal C_\gamma$.
\item[$2.$]
The topology of $\mathcal C_\gamma$ induced from the intrinsic metric $\delta$ $($see definition below$)$
 coincides with the induced topology of $\mathcal C_\gamma$ from~$(M,F)$.
\item[$3.$]
The metric space $\mathcal C_\gamma$ with the intrinsic metric $\delta$ is forward complete.
\end{enumerate}
\end{Theorem}

Indeed, by the f\/irst statement, any two co-points $q_1,q_2\in \mathcal C_\gamma$ can be joined by a rectif\/iable arc in $\mathcal C_\gamma$
if $q_1$ and $q_2$ are in the same connected component.
Therefore, the {\it intrinsic metric}~$\delta$ on~$\mathcal C_\gamma$ def\/ined as
\begin{gather*}
\delta(q_1,q_2):=
\begin{cases}
\inf\{l(c)|\ c\ \textrm{is a rectif\/iable arc in }\mathcal C_\gamma\ \textrm{joining } q_1\ \textrm{and } q_2\},\\
\qquad\textrm{if $q_1,q_2\in \mathcal C_\gamma$ are in the same connected component,}\\
+\infty, \qquad \textrm{otherwise}
\end{cases}
\end{gather*}
is well def\/ined.

By Theorem~\ref{th1.7} and Theorem~C in~\cite{TS} we have
\begin{Theorem}\label{th1.12}
Let $\gamma$ be a ray in a forward complete boundaryless $2$-dimensional Finsler mani\-fold~$(M,F)$.
Then, there exists a set ${\cal E}\subset [0,\infty)$ of measure zero with the following properties:
\begin{enumerate}
\item[$1.$] For each
$t\in(0,\infty){\setminus}{\cal E}$, the set  $\bb_\gamma^{-1}(t)$ consists of locally finitely many mutually disjoint arcs.
In particular, if $\bb_\gamma^{-1}(a)$, is compact for some $a>t$, then $\bb_\gamma^{-1}(t)$ consists of finitely many mutually disjoint circles.
\item[$2.$] For each $t\in(0,\infty){\setminus} {\cal E}$, any point $q\in \bb_\gamma^{-1}(t)$ admits at most two maximal co-rays.
\end{enumerate}
\end{Theorem}

Here {\it locally finitely many} means that for $x\in \bb_\gamma(t)$, and any forward (or backward) ball $\mathcal B^+(x,r):=\{p\in M\colon d(x,p)<r\}$, the set $\mathcal B^+(x,r)\cap \bb_\gamma(t)$ contains only f\/initely many arcs.

\section[Implications of the differentiability of $\bb_\gamma$]{Implications of the dif\/ferentiability of $\boldsymbol{\bb_\gamma}$}\label{sec: diff of Bus}

Here are some results that follow from the previous section (compare with \cite{In1}).

 In~\cite{In1} it is proved for $G$-spaces that
if the co-point set ${\mathcal C}_\gamma\neq \varnothing$ is compact, then $\bb_\gamma$ is an exhaustion function.
We will give a more general result.

\begin{Theorem}\label{th2.2}
Let $(M,F)$ be a forward complete non-compact boundaryless Finsler manifold and $\gamma$ a ray in $M$.

 If for some given $c\geq \inf \bb_\gamma(M)$,
the set
 $\mathcal C_\gamma\cap \bb_\gamma^{-1}(-\infty,c]$ is compact and non-empty,
 then $\bb^{-1}_\gamma(-\infty,c]\neq \varnothing$ is compact.
\end{Theorem}
\begin{proof}

For the number $c\geq \inf \bb_\gamma(M)$ given, we def\/ine the set
\begin{gather*}
S_c:=\big\{q\in \bb_\gamma^{-1}(c)\, |\, q \textrm{ belongs to some co-ray to }\gamma \textrm{ emanating from a point in }\mathcal C_\gamma\big\}.
\end{gather*}

For later use we also def\/ine
\begin{gather*}
\widetilde{M}:=\{ x\in M\, |\, \textrm{there is a maximal co-ray }  \sigma_x\colon (a,\infty)\to M, \textrm{ passing through } x, \\
 \hphantom{\widetilde{M}:=\{}{} \textrm{ for some } a\in [-\infty,\infty) \}.
\end{gather*}

Remark that if $F$ is bi-complete, then $a=-\infty$ always in the def\/inition of $\widetilde M$, but since we assume only forward completeness here, a f\/inite value for $a$ might happen.

We will divide the proof in two steps.

{\bf Step 1.}
Firstly, we prove that
\begin{gather*}
S_c=\bb_\gamma^{-1}(c).
\end{gather*}

In the case $c< \inf \bb_\gamma(M)$, we prove this by showing the followings
\begin{enumerate}\itemsep=0pt
\item[(i)] the set $S_c$ is non-empty,
\item[(ii)] $S_c$ is open in $\bb_\gamma^{-1}(c)$,
\item[(iii)] $\widetilde M$ is closed in $M$,
\item[(iv)] $\widetilde M$ is open in $M$.
\end{enumerate}

 {\it Proof of} (i).\ Firstly, we show that $S_c\neq \varnothing$. Indeed, taking into account the hypothesis, we can consider a point $p\in \mathcal C_\gamma\cap \bb_\gamma^{-1}(-\infty,c]$. If $\bb_\gamma(p)=c$ then $p\in S_c$ and there is nothing to prove.

 We can therefore assume $\bb_\gamma(p)<c$, that is there exists $l>0$ such that $\bb_\gamma(p)=c-l$.

 Since $p\in \mathcal C_\gamma$, we consider the maximal co-ray $\sigma\colon [0,\infty)\to M$ from $p$ to $\gamma$ and let $q$ be the point on $\sigma$ such that
 $d(p,q)=l$. Then $\bb_\gamma(\sigma(s))=s+\bb_\gamma(p)$ implies $\bb_\gamma(q)=d(p,q)+ \bb_\gamma(p)=d(p,q)+c-l=c$ and hence $q\in \bb_\gamma^{-1}(c)$ and $q\in \sigma$, that is $q\in S_c$. These show that $S_c$ is non-empty and (i) is proved.

{\it Proof of} (ii).\ Next, we prove by contradiction that $S_c$ is open. Indeed, assuming by contradiction
that for $q\in S_{c}$
 there is a points sequence $\{q_j\}\subset \bb_\gamma^{-1}(c){\setminus} S_c$ such that
 $q=\lim\limits_{j\to \infty}q_j$. We denote by $\sigma_j$ and $\sigma$ the co-rays passing through $q_j$ and $q$, respectively. Let $x$ be the initial point of $\sigma$, and by our assumption $x\in \mathcal C_\gamma\cap \bb_\gamma^{-1}(-\infty,c]$.

 Consider now a scalar $r>d(q,x)$ and the forward closed ball \mbox{$\overline{\mathcal B^+(q,r)}:=\{p\!\in\! M\, |\, d(q,p)\leq r\}$}. Obviously $\overline{\mathcal B^+(q,r)}$ is compact due to the forward complete hypothesis and the Hopf--Rinow theo\-rem, and $x\in \overline{\mathcal B^+(q,r)}$.

Let $\sigma_j$ denote a co-ray to $\gamma$ emanating from
$q_j=\sigma_j(0)$.
Since $\overline{{\mathcal B}^+(q,r)}$ is compact and $q_j\notin S_c$, we can extend backward $\sigma_j$ to some interval
$[s_j,0]$ with $d(q,\sigma_j(s_j))=r+\delta$.
Any limit geodesic of the sequence
$\{\sigma_j\}$ is a co-ray passing through $q$ which
contains $x$ as an interior point, that is a contradiction (see Fig.~\ref{S_i is open.}).

It follows that $S_c$ must be open set and (ii) is proved.

\begin{figure}[h]
\centering
\includegraphics{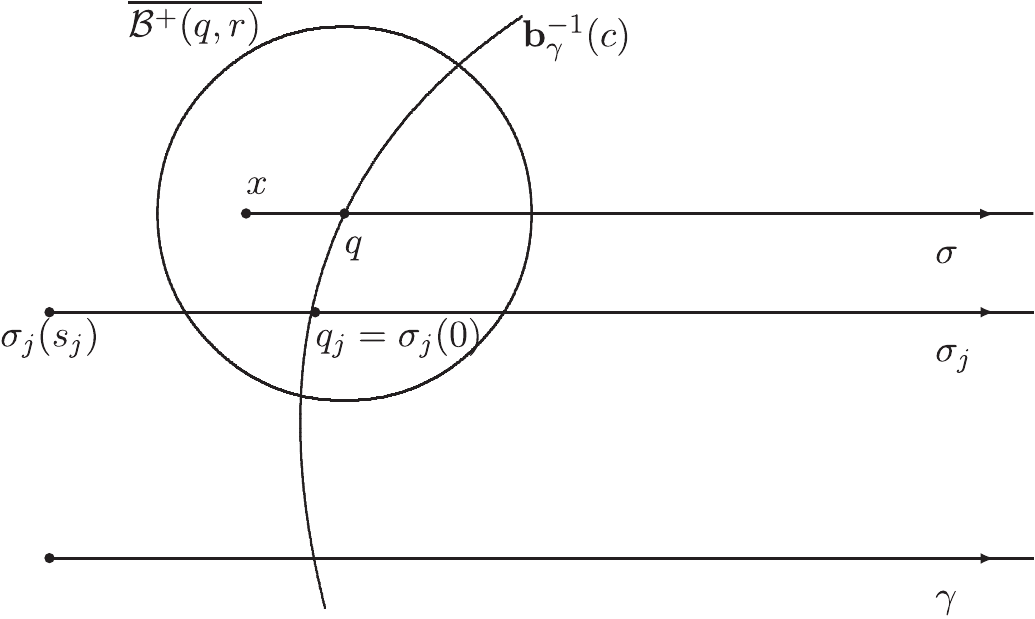}

\caption{$S_c$ is an open set.}\label{S_i is open.}
\end{figure}

{\it Proof of} (iii).\ Now we show that $\widetilde M$ is closed in $M$. Indeed let $\{x_j\}$ be a set of points in $\widetilde M$, such that $x_j\to x$ in $M$, and let $\sigma_j\colon (a_j,\infty)\to M$ be maximal co-rays to $\gamma$, parametrised such that $\sigma_j(0)=x_j$, with $a_j\in [-\infty,0)$. Obviously, such $\sigma_j$ exist from the def\/inition of $\widetilde M$.

By extracting some sub-sequence of $\sigma_j$ we can assume, without loosing the generality, that $\sigma_j |_{[0,\infty)}$ converges to some $\sigma_x |_{[0,\infty)}$.

We will show that $x\in \widetilde{M}$. As the sequence $\{x_j\}$ was arbitrary, this would imply $\widetilde{M}$ closed. So assume by contradiction
$x\in M{\setminus} \widetilde M$. This means that the domain of def\/inition of the maximal co-ray extension of $\sigma_x$ is the interval $[a,\infty)$, for some $a\in [-\infty,0)$. Therefore $\sigma_x(a)\in \mathcal C_\gamma$, and hence $q:=\sigma_x(1)\in S_c$, for $c:=1+\bb_\gamma(x)$, where we use Theorem~\ref{th1.4}.

On the other hand, for $j$ large enough, consider the points $q_j:=\sigma_j|_{[0,\infty)}\cap \bb_\gamma^{-1}(c)$, and observe that $q_j\in \bb_\gamma^{-1}(c){\setminus} S_c$, by def\/inition. But this contradicts the fact that $S_c$ is open, a fact proved already in (ii).

Thus, the (unique) maximal co-ray through $x$ must be of the form $\sigma_x\colon (a,\infty)\to M$, for some $a\in [-\infty,0)$, and therefore $x\in \widetilde M$. This shows that $
\widetilde M$ is closed and hence (iii) is proved.

{\it Proof of} (iv).\ Finally, we prove $\widetilde M$ is open set in $M$, or, equivalently, that $\widetilde M{\setminus} M$ is closed, in a similar manner.

Consider a sequence $\{x_j\}$ in $\widetilde M{\setminus} M$ with $x_j\to x$ in $M$, and consider the maximal co-rays
$\sigma_j\colon [a_j,\infty)\to M$, with $\sigma_j(0)=x_j$. Obviously this is the form of the maximal co-rays due to the choice of $x_j$ and def\/inition of $\widetilde M$.

Observe that the sequence of points $\{\sigma_j(a_j)\}\subset \mathcal C_\gamma\cap \bb_\gamma^{-1}(-\infty,c]$, for $c:=1+\bb_\gamma(1)$, and by the compactness hypothesis of $ \mathcal C_\gamma\cap \bb_\gamma^{-1}(-\infty,c]$ it follows that the limit point $\sigma_x(a)\in \mathcal C\cap \bb_\gamma^{-1}(-\infty,c]$. Thus
$x\in \widetilde M{\setminus} M$, and hence $\widetilde M$ is open, so (iv) is proved.

Using these we will build our argument as follows.

Reminding ourselves that a topological space $X$ is connected if and only if the only closed and open sets are the empty set and $X$, by taking $X=M$, and using claims (iii) and (iv) proved above, that is $\widetilde M$ is closed and open in $M$ it follows $\widetilde M=\varnothing$ or $\widetilde M=M$ (obviously $M$ is connected by hypothesis). However, since $S_c\neq \varnothing$ (claim (i) proved above), the maximal co-ray passing through any point $x\in S_c$ has the form $\sigma_x\colon [a,\infty)\to M$, so $\widetilde M=M$ cannot be possible, hence $\widetilde M=\varnothing$.

We obtain $S_c=\bb_\gamma^{-1}(c)$. Indeed, $S_c\subset\bb_\gamma^{-1}(c)$ by def\/inition. Conversely, for any $q\in \bb_\gamma^{-1}(c)$, it is now clear that there exists a maximal co-ray $\sigma_q\colon [a,\infty)\to M$ to $\gamma$ through $q$, hence $q\in S_c$ and
the claim $S=\bb_\gamma^{-1}(c)$ is proved.

We discuss now the case $c=\inf \bb_{\gamma}(M)$. Firstly, we observe that, for any $c\geq \inf \bb_{\gamma}(M)$, if $\sigma\colon (-\ve,\infty)\to M$ is a co-ray such that $\sigma(0)\in \bb_{\gamma}^{-1}(c)$, then $\bb_{\gamma}(\sigma(-\frac{\ve}{2}))<c$, and thus $\bb_{\gamma}^{-1}( \inf \bb_{\gamma}(M))\cap\widetilde{M}=\varnothing$.

In particular, if $c=\inf \bb_{\gamma}(M)$, and $\bb_{\gamma}^{-1}(-\infty,c]=\bb_{\gamma}^{-1}(c)$ is non-empty, then $S_{c}=\bb_{\gamma}^{-1}(c)$. The proof of this fact is similar to the proof of (i).

Observe that this immediately implies that $S_c\neq \varnothing$ for all $c\geq \inf \bb_{\gamma}(M)$.

In other words, what we have proved in Step 1 is that
{\it for any point $q\in \bb_\gamma^{-1}(c)$, there exists a maximal co-ray, i.e., a co-ray emanating from a point $x\in \mathcal C_\gamma$, passing through~$q$.}

{\bf Step 2.}
Using this we proceed to proving that $\bb_\gamma^{-1}(-\infty,c]$ is compact. We assume the converse, i.e., we assume there exists a divergent sequence
$\{x_j\}$ in $\bb_\gamma^{-1}(-\infty,c]$ in the sense that for any compact set $K\subset \bb_\gamma^{-1}(-\infty,c]$, there exists $N_K>0$ such that $x_j\notin K$ for any $j>N_K$.

 For each $j$ there exists a co-ray $\sigma_j$ from $x_j$ such that $\sigma_j\cap \bb_\gamma^{-1}(c)=\{y_j\}$ (this can be easily seen
 by a similar argument as in the proof of $S\neq \varnothing$). From Step 1 it follows that we can extend~$\sigma_j$ up to the point $z_j=\sigma_j(0)\in \mathcal C_\gamma$.

From hypothesis 2 of the theorem, there exists a subsequence $z_{j_k}$ of $z_j$ convergent to $z$ and hence there exists a sequence of co-rays $\sigma_{j_k}$ (emanating from each $z_{j_k}$) convergent to a co-ray $\sigma$ (emanating from the limit point $z$). For the sake of simplicity we assume $\lim\limits_{j\to \infty}z_j=z$. It follows that there exists a point $y\in \bb_\gamma^{-1}(c)$ such that $\lim\limits_{j\to \infty}y_j=y$.

Since $x_j$ is by construction an interior point of the
$N_\gamma^{c}$-segment $\sigma_j|_{[0,s_j]}$ that joins $z_j$ to $y_j$, it follows that there exists a point $x$ interior to the
$N_\gamma^{c}$-segment $\sigma|_{[0,s]}$ that joins $z$ to $y$.
But this implies that the sequence $\{x_j\}$ cannot be divergent in the sense above, that is we obtain a~contradiction. Therefore, $\bb_\gamma^{-1}(-\infty,c]$ must be compact.
 \end{proof}

\begin{Remark}\quad
\begin{enumerate}\itemsep=0pt
\item Observe that the conclusion of the theorem above cannot hold for $c< \inf \bb_{\gamma}(M)$ since, in this case, $\bb_\gamma^{-1}(-\infty,c]$ would be empty set.
 \item
Similarly with the proof of (ii) above, one can show that actually $S_c$ is also closed. We will not prove this property here because we don't need it.
\end{enumerate}
\end{Remark}

\begin{Corollary}\label{exhaust}
Let $(M,F)$ be a forward complete non-compact Finsler manifold and $\gamma$ a ray in~$M$.
 If there exists a numerical sequence $\{c_i\}$ with $\lim\limits_{i\to \infty}c_i=+\infty$, such that for each~$i$
 such that
 $\mathcal C_\gamma\cap \bb_\gamma^{-1}(-\infty,c_i]$ is compact and non-empty,
 then set $\bb^{-1}_\gamma(-\infty,c_i]$ is compact, i.e.,
the Busemann function $\bb_\gamma$ is an exhaustion function.
\end{Corollary}

The following lemma shows that Innami's result in~\cite{In1} is a special case of
our Theorem~\ref{th2.2}.

\begin{Lemma}\label{lem2.3}
 Let $(M,F)$ be a bi-complete Finsler manifold and $\gamma$ a ray in $M$. If $\mathcal C_\gamma\neq \varnothing$ is compact, then for all sufficiently large $a\in \R$, the level set $\bb_\gamma^{-1}(a)$ is arcwise connected.
\end{Lemma}
\begin{proof}
Since $\mathcal C_\gamma\neq \varnothing$ is compact we can choose a number $a>\max \bb_\gamma({\mathcal C}_\gamma)$.
Thus there does not exist a co-point of $\gamma $ in $\bb_\gamma^{-1}[a,\infty)$.
Choose any two points $x$ and $y$ in $\bb_\gamma^{-1}(a)$.
By Lemma~\ref{lem1.1},
there exists a continuous curve $c$ in $\bb_\gamma^{-1}[a,\infty)$ joining $x$ to $y$.
Since ${\mathcal C}_\gamma\cap \bb_\gamma^{-1}[a,\infty)=\varnothing$, we can get a curve in
$\bb_\gamma^{-1}(a)$ joining $x$ to $y$ by deforming the curve $c$ along the co-rays intersecting~$c$.
Therefore, the level set is arcwise connected.
\end{proof}

\begin{Remark}
Observe that the bi-completeness hypothesis is needed for deforming the curve $c$ above.
\end{Remark}

Moreover, we have

 \begin{Theorem}\label{MTF2}
 Let $(M,F)$ be a forward complete non-compact boundaryless Finsler manifold and $\gamma$ a ray in $M$.
 If $\mathcal C_\gamma$ contains an isolated point $p$, then
 \begin{enumerate}\itemsep=0pt
 \item[$1.$] The exponential map $\exp_p\colon T_pM\to M$ is a $C^1$-diffeomorphism and any geodesic emanating from $p$ is a maximal co-ray to $\gamma$.
 \item[$2.$] $\mathcal C_\gamma=\{p\}$ only.
 \item[$3.$] For any fixed point $q\in M$, the relation
 \begin{gather}
 d(p,q)+\bb_\gamma(p)=\bb_\gamma(q)
 \end{gather}
 holds good. In particular, for any
 $a>\bb_\gamma(p)$ the level sets $\bb^{-1}_\gamma(a)$ coincide with the forward spheres $\mathcal S^+(p,a-\bb_\gamma(p)):=\{q\in M\colon d(p,q)=a-\bb_\gamma(p)\}$.
 \end{enumerate}
 \end{Theorem}

 \begin{proof}
 1. Since $p\in \mathcal{C}_\gamma$ is isolated in $\mathcal{C}_\gamma$, it follows by def\/inition that there exists $\varepsilon_0>0$ such that
 \begin{gather*}
 \mathcal{C}_\gamma\cap \mathcal B^+(p,\varepsilon_0)=\{p\},
 \end{gather*}
 where $\mathcal B^+(p,\varepsilon_0)$ is the forward ball in $(M,F)$.

 {\bf Claim 1.} {\it There exists $\varepsilon_1\in (0,\varepsilon_0)$ such that any co-ray to $\gamma$ emanating from a point of $M{\setminus} \mathcal B^+(p,\varepsilon_0)$ does not intersect $\mathcal B^+(p,\varepsilon_1)$.}

 Indeed, let us assume the contrary, that is, we shall assume that for each positive integer $j\in\{1,2,\dots\}$, there exists a co-ray
 $\sigma_j\colon [0,\infty) \to M$, emanating from a point $q_j=\sigma_j(0)\notin \mathcal B^+(p,\varepsilon_0)$, that intersects
 $\mathcal B^+(p,\frac{1}{j})$.

 Under this assumption, by extracting a subsequence
 of $\{\sigma_j\}$ we can construct a convergent sequence of co-rays with the properties in the assumption above. For simplicity, we denote this subsequence by $\{\sigma_j\}$ again. In this way, we obtain a limit co-ray $\sigma:=\lim\limits_{j\to \infty}\sigma_j$, and a convergent sequence of points
 $p_j\in \mathcal B^+(p,\frac{1}{j}) \cap \sigma_j|_{[0,\infty)}$, $\lim\limits_{j\to \infty}p_j=p$. It follows that there exists a co-ray $\sigma$ and $p$ is interior point of $\sigma$. This is a contradiction with $p\in \mathcal C_\gamma$ and Claim~1 is proved.

 {\bf Claim 2.} {\it For any point $q\in \mathcal B^+(p,\varepsilon_1)$, the geodesic emanating from $p$ and passing through the point $q$ is a co-ray of~$\gamma$. }

 Let $\sigma\colon [0,\infty)\to M$ be a co-ray to $\gamma$ emanating from $q=\sigma(0)$, and let $\widetilde\sigma\colon (a,\infty)\to M$ be the maximal geodesic extension of $\sigma$. One of the following situations happen.

 {\it Case 1. } $\widetilde\sigma|_{(a,0]}\subset \mathcal B^+(p,\varepsilon_0)$.

 In this case, since $\mathcal B^+(p,\varepsilon_0)$ is compact, there exists $b\in (a,0)$ such that $\widetilde\sigma|_{[b,0]}$ is not minimal. In particular, $\widetilde\sigma|_{[b,\infty)}$ is not a ray.

Thus, there must exist $b_1\in (b,0]$ such that $\widetilde\sigma|_{[b_1,\infty)}$ is a maximal co-ray to $\gamma$. It follows
$\widetilde\sigma(b_1)\in \mathcal C_\gamma\cap \mathcal B^+(p,\varepsilon_0)=\{p\} $ and hence $\widetilde\sigma|_{[b_1,\infty)}$ is a co-ray to $\gamma$ emanating from $p$ and passing through the point $q$.

 {\it Case 2. } There exists $b\in (a,0]$ such that $\widetilde\sigma(b) \notin \mathcal B^+(p,\varepsilon_0)$.

 In this case, let us denote $b_1:=\max\{t<0\colon d(p,\widetilde{\sigma}(t))=\varepsilon_0\}$. Since $\widetilde{\sigma}(0)=q\in \mathcal B^+(p,\varepsilon_1)$, it results that $\widetilde\sigma|_{[b_1,\infty)}$ is not a co-ray to $\gamma$ and $\widetilde{\sigma}|_{(b_1,0)}\subset \mathcal B^+(p,\varepsilon_0)$.

 Therefore, it must exist $b_2\in (b_1,0)$ such that $\widetilde{\sigma}|_{[b_2,\infty)}$ is a maximal co-ray to $\gamma$ passing through
 $\widetilde\sigma(b_2)\in \mathcal C_\gamma\cap \mathcal B^+(p,\varepsilon_0)=\{p\} $.

 From these it results that for any point $q\in \mathcal B^+(p,\varepsilon_1)$, there exists a co-ray to $\gamma$ emanating from $p$ and passing through the point $q$ and Claim 2 is proved.

 From Claims 1 and 2 we conclude that any geodesic from $p$ is a co-ray to~$\gamma$.

 It follows now from Theorem~\ref{th1.7} that any two distinct co-rays of $\gamma$ emanating from~$p$ do not intersect again. Indeed, it is trivial to see that since $\bb_\gamma$ is dif\/ferentiable at interior points of co-rays and the tangent direction of the co-ray at such a point is $\nabla\bb_\gamma$, it is not possible for co-rays to intersect each other at their interior points.

 In this way we obtain that $\exp_p\colon T_pM\to M$ is a $C^1$-dif\/feomorphism and f\/irst statement of the theorem is proved.

2. The fact that $C_\gamma=\{p\}$ it is now obvious from the proof of 1.

 3. Let us choose any point $q\in M{\setminus}\{p\}$, and let us denote by $\beta\colon [0,\infty )\to M$ any unit speed geodesic emanating from $p$ and passing through the point $q$. From the f\/irst statement of this theorem it follows that $\beta$ must be a co-ray to $\gamma$ and hence the relation
 $\bb_\gamma(\beta(s))=s+\bb_\gamma(p)$
 holds for any $s\geq 0$. In particular, since $q=\beta(d(p,q))$, it results $\bb_\gamma(q)=d(p,q)+\bb_\gamma(p)$.

 Moreover, from here it follows that for any $a>\bb_\gamma(p)$, we have $\bb_\gamma^{-1}(a)=\mathcal S^+(p,a-\bb_\gamma(a)) $, and the theorem is proved.
 \end{proof}

 \begin{Remark}
 If $(M,F)$ is a non-compact Finsler manifold whose all geodesics are strightlines, then $(M,F)$ is bi-complete and $\mathcal C_\gamma=\varnothing$.
 \end{Remark}

\begin{Remark}
It would be interesting to obtain some geometrical conditions (f\/lag curvature conditions) on the Finsler manifold $(M,F)$ such that all Busemann functions are everywhere dif\/ferentiable. Since this topic requires more elaboration, we leave it for a future research.
\end{Remark}

We recall that an {\it end} $\varepsilon$ of a non-compact manifold $X$ is an assignment to each compact set $K\subset X$ a component $\varepsilon(K)$ of $X{\setminus} K$ such that $\varepsilon(K_1)\supset\varepsilon(K_2)$ if $K_1\subset K_2$. Every non-compact manifold has at least one end. For instance, $\R^{n}$ has one end if $n>1$ and two ends if $n=1$.
By def\/inition one can see that a product $\R\times N$ has one end if $N$ is non-compact and two ends otherwise.

Here we prove
\begin{Corollary}\label{cor: topological decomp}
Let $(M,F)$ be a bi-complete non-compact Finsler manifold.
\begin{enumerate}\itemsep=0pt
\item[$1.$] If $\mathcal C_\gamma=\varnothing$, then $M$ is homeomorphic to $\R\times \bb_\gamma^{-1}(0)$.
\item[$2.$] If $M$ has at least three ends, then there are no differentiable Busemann functions on~$M$.
\end{enumerate}
\end{Corollary}

\begin{proof} 1. Since $\mathcal C_\gamma=\varnothing$, it follows that $\bb_\gamma$ is smooth everywhere and hence from each point there is a unique co-ray to $\gamma$. Thus, we can def\/ine the function
$\varphi\colon M\to \R\times \bb_\gamma^{-1}(0)$, $p\mapsto \varphi(\bb_\gamma(p),h_1(p))$,
where $h_1(p)$ is the intersection point of the co-ray from $p$ with the level set
$\bb_\gamma^{-1}(0)$. From the bi-completness hypothesis it follows that $h_1(p)$ always exists. We can easily see that this is a homeomorphism by using the Lipschitz continuity of~$\bb_\gamma$.

2. Due to statement 1 it follows that if $\bb_\gamma$ is dif\/ferentiable, then $M$ have at most two ends. Statement~2 follows by logical negation.
\end{proof}

\begin{Remark}\label{rm: not closed}
It is known that the cut locus of a point in a Riemannian or Finsler manifold~$M$ is a closed subset of~$M$ (see~\cite{BCS}).
On the other hand, we have shown in~\cite{TS}, by an example, that the cut locus of a closed subset in~$M$ is not closed in~$M$ anymore. A natural question is if the co-point set $\mathcal C_\gamma$ is closed or not.
First to answer to this question is Nasu who constructed in~\cite{N1} an example of Riemannian complete surface with a ray $\gamma$ whose co-point set $\mathcal C_\gamma$ is {\it not closed}.
Obviously, the same conclusion can be derived from our Theorem~\ref{th1.7} and~\cite{TS}.
\end{Remark}

\subsection*{Acknowledgements}

 I am grateful to Professor M.~Tanaka
 for bringing this topic into my attention and for many
 illuminating discussions. I am also deeply indebted to the anonymous referees for their constructive criticism and extremely useful suggestions that improved the manuscript enormously.
 Also I thank to N.~Boonnam for reading an early version of the paper.

\pdfbookmark[1]{References}{ref}
\LastPageEnding

\end{document}